\documentclass[12pt,a4paper]{article}

\usepackage{amsthm}
\usepackage{amsmath}
\usepackage{amssymb}
\usepackage{fullpage}

\newcommand{\N}{\mathbb{N}}
\newcommand{\Z}{\mathbb{Z}}
\newcommand{\R}{\mathbb{R}}

\newcommand{\RR}{\mathcal{R}}
\newcommand{\RRand}{\mathcal{\RR}_{\textsf{RAND}}}
\newcommand{\eps}{\epsilon}
\newcommand{\half}{\frac{1}{2}}
\newcommand{\seq}{\subseteq}

\newcommand{\norm}[1]{\left\| #1 \right\|}
\newcommand{\E}{{\mathbb E}}
\newcommand{\girth}{\textrm{girth}}

\newtheorem{theorem}{Theorem}[section]
\newtheorem{corollary}[theorem]{Corollary}
\newtheorem{proposition}[theorem]{Proposition}

\newtheorem{lemma}[theorem]{Lemma}
\newtheorem{claim}[theorem]{Claim}
\newtheorem{observation}[theorem]{Observation}

\newtheorem{conjecture}[theorem]{Conjecture}

\newtheorem{remark}[theorem]{Remark}

\begin{document}

\title{Greedy Random Walk}
\author{
    Tal Orenshtein
    \thanks{Department of Mathematics,
      Weizmann Institute of Science, Rehovot, {\sc Israel}.
      {\tt tal.orenshtein@weizmann.ac.il}}
    \and
    Igor Shinkar
    \thanks{Department of Computer Science,
      Weizmann Institute of Science, Rehovot, {\sc Israel}.
      {\tt igor.shinkar@weizmann.ac.il}}
}
\date{}

\maketitle

\abstract
We study a discrete time self interacting
random process on graphs, which we call Greedy Random Walk.
The walker is located initially at some vertex. As time evolves,
each vertex maintains the set of adjacent edges touching it that have not been
crossed yet by the walker. At each step, the walker being at some vertex,
picks an adjacent edge among the edges that have not traversed
thus far according to some (deterministic or randomized) rule.
If all the adjacent edges have already been traversed,
then an adjacent edge is chosen uniformly at random. After picking an
edge the walk jumps along it to the neighboring vertex.
We show that the expected edge cover time of the greedy random walk is
linear in the number of edges for certain natural families
of graphs. Examples of such graphs include the complete graph,
even degree expanders of logarithmic girth, and the hypercube graph.
We also show that GRW is transient in $\Z^d$ for all $d \geq 3$.

\section{Introduction}\label{sec:intro}

    {\it Greedy Random Walk} (GRW) on a graph is a discrete time
    random process, with transition law defined as follows.
    The walker is located initially at some vertex of the graph.
    As time evolves each vertex in the graph maintains the set of all
    adjacent edges that the walker has not crossed yet.
    At each step the walker picks an unvisited edge among the edges
    adjacent to its current location arbitrarily according to some rule.
    If all the adjacent edges have already been visited, an adjacent
    edge is picked uniformly at random.
    The walker, then, jumps to a neighboring vertex along the chosen edge.
    We think of the process as trying to cover the graph as fast as possible
    by using a greedy rule that prefers to walk along an unvisited edge whenever possible.
    This suggests the name {\em Greedy Random Walk}.
    \medskip

    Formally, for an undirected graph $G=(V,E)$ a GRW with a (possibly randomized)
    rule $\RR$ on $G$ is a sequence $X_0,X_1,X_2, \dots$ of random variables defined on $V$
    with the following transition probabilities.
    For each $t \geq 0$ define
    \begin{equation}\label{eq:H_t}
        H_t = \{ (X_{s-1},X_{s}) \in E : 0 < s \leq t\}
    \end{equation}
    to be the set of all the edges traversed by the walk up to time $t$.
    For every vertex $v \in V$ and time $t \geq 0$ define
    \begin{equation}\label{eq:J_t(v)}
        J_t(v) = \{e \in E : v \in e \mbox{ and } e \notin H_t\}
    \end{equation}
    to be the set of all the edges touching $v$ that have not been traversed by the walk up to time $t$.
    Denoting by $N_v$ is the set of neighbors of $v$ in $G$,
    the transition probabilities are given by:
    \[
        \Pr[X_{t+1} = w | ( X_i )_{i \leq t} ] =
            \begin{cases}
                \RR( w | ( X_i )_{i \leq t} )    &   J_t(X_t) \neq \emptyset \mbox{ and } \{X_t,w\} \in J_t(X_t) \\
                \frac{1}{|N_{X_t}|}     &   J_t(X_t) = \emptyset \mbox{ and } w \in N_{X_t} \\
                0                       &   \mbox{ otherwise },
            \end{cases}
    \]
    where $\RR( w | ( X_i )_{i \leq t} )$ denotes the probability of choosing $w \in N_{X_t}$
    conditioned on the information regarding the process so far.
    A natural rule $\RR$ is to choose uniformly at random an edge among
    the adjacent unvisited edges $J_t(v)$ of the current vertex $v=X_t$.
    We shall denote this rule by $\RRand$.
    \medskip

    One can think of GRW as a random walk where the walker wishes to cover
    the graph as fast as possible and is allowed to make some local computation
    at each vertex she visits (e.g., mark the last edge that the walker used
    to reach the current vertex, and also mark the edge that the walker is
    going to use in the next step), but is not allowed to transfer information
    between vertices. A motivation for the study of GRW arises from distributed
    computation in which an agent sits on every vertex of a graph.
    Each agent has a list of neighbors and is allowed to communicate only with them.
    The goal is to let all the agents use some resource as fast as possible,
    while using only the local information for each vertex, and no extra information
    regarding the graph and the vertices that have already been visited.
    An agent has a list of neighbors who communicated with him thus far during the process,
    and each time the agent receives the resource, she is allowed to perform only local
    computations before moving it to one of her neighbors.
    We will see that the GRW protocol performs better than simple random
    walk (SRW) on some families of graphs.

    The main difficulty in analyzing such random process comes from the fact that
    GRW is self-interacting, i.e., is not a Markov chain (meaning that the
    probability distribution of the next step depends not only on the current
    position of the walker, but also on the entire walk thus far). Although
    in many cases a certain property of self interacting random walks
    can be observed in simulations or seems to be
    suggested by ``heuristical proof'', typically it is much harder to give robust
    proofs for random walks that do not have the Markov property.
    Related models include
    RW with choice~\cite{AK06}
    non-backtracking RW~\cite{ABLS06},
    RW with neighborhood exploration~\cite{BCERS},
    excited RW~\cite{BW03}
    reinforced RW~\cite{Pem07},
    rotor router RW~\cite{FrSa10},
    and more.
    Recently this model has been considered independently by Berenbrink et al.~\cite{BCF}.
    They showed that if $G$ is an even degree expander graph such that
    every vertex is contained in a vertex-induced cycle of logarithmic length,
    then the expected vertex cover time by GRW is linear for any rule $\RR$.

\subsection{Our Results}\label{sec:our results}

    In Section~\ref{sec:edge_cover_time} we study the edge cover time of GRW on finite graphs.
    Obviously, the edge cover time of any graph $G=(V,E)$ is at least $|E|$,
    as the walker must cross every edge at least once.
    We prove bounds on the edge cover time of GRW
    by analyzing the ``overhead'' of the walk, i.e.,
    the difference between the expected edge cover time of the walk,
    and the number of edges in a graph.
    For example, we establish that the expected time it takes
    for GRW to go via all edges of $K_n$, the complete graph on $n$ vertices
    is ${n \choose 2} + (1 + o(1))n \log(n)$. Therefore, the aforementioned
    ``overhead'' in the case of $K_n$ is $(1 + o(1))n \log(n)$.
    In particular, all edges of $K_n$ are covered by GRW in time
    is $(1 + o(1)) \cdot {n \choose 2}$, which is asymptotically faster
    than $\Theta(n^2 \log n)$, the expected edge cover time of SRW.

    We show that for certain families of graphs the expected
    edge cover time of GRW is asymptotically faster than that of SRW.
    In particular, we establish that
    expected edge cover time of GRW is \emph{linear} in the number of edges
    for the complete graph, for the hypercube graph,
    and for constant even degree expanders with logarithmic girth.
    The later result is claimed in the paper of Berenbrink et al.~\cite{BCF}
    without proof.

    Another interesting result is given in Lemma~\ref{lemma:SRW cover time bound}
    that bounds the \emph{edge} cover time of an even degree graph by GRW
    in terms of its \emph{vertex} cover time by SRW. Specifically, we show
    that for any graph $G = (V,E)$ whose vertices have even degrees,
    and its expected \emph{vertex} cover time by SRW is $C$ it holds that
    the expected \emph{edge} cover time of $G$ using GRW is at most $|E| + C$.
    Therefore, for even degree graphs of logarithmic degree, whose vertex cover time is $O(n \log(n))$
    we obtain a bound on the edge cover time which is linear in the number of edges.

    These result should be compared with the general lower bound
    on the expected cover time of graphs by SRW.
    Recall that Feige \cite{Feige95lowerbound} has shown that
    for any graph with $n$ vertices the expected vertex cover time
    by a simple random walk is at least $(1 - o(1))n \log n$.
    Analogously, for all graphs the expected edge cover cover
    is at least $\Omega(|E| \log(|E|))$ (see \cite{Zuckerman90lowerbound}, \cite{aldous1989lowerbd}).
    In this direction, a result of Benjamini, Gurel-Gurevich and Morris \cite{BGM}
    says that for bounded degree graphs linear cover time is exponentially unlikely.

    We are also interested in the behavior of GRW on infinite graphs.
    It is well known that SRW on $\Z^d$ is transient if $d \geq 3$, and recurrent otherwise.
    We prove that GRW is transient on $\Z^d$ for $d \geq 3$.
    The case of $d=2$ remains open, and it is shown to be equivalent the notorious
    two dimensional mirror model problem \cite{Ruijgrok,BT92}.
    Our proof holds for all graphs with even degrees on which SRW is transient.
    This leaves unsolved the question of transience of GRW in latices with odd degrees.
    These and other related results are discussed in Section~\ref{sec:Z^d}, which can be read independently
    of the rest of the paper.

    \paragraph{General remarks:}
    \begin{description}
    \item[The choice of the rule $\RR$:]
        In the first version of this paper we considered GRW that uses only the rule $\RRand$.
        After the first version of our work was uploaded to \texttt{arxiv.org},
        Berenbrink et al.~\cite{BCF} independently published their work in which
        they consider GRW with \emph{any} (deterministic or randomized) rule,
        even adversarial ones that try to slow-down the process.
        After reading their results we have noticed that in fact our
        proofs for upper bounding the edge cover time are independent
        of $\RR$ and hold for any rule as well.
    \item[The choice of the starting vertex:]
        In all of our results on cover time the bounds are independent of the
        starting vertex. Also, in most cases the considered graphs are
        vertex-transitive, and so, specification of the starting vertex is unnecessary.
    \end{description}

\subsection{Notation}\label{sec:notation}
    We use the standard notations of asymptotic growth rates.
    For two functions $f,g : \N \to \R_+$
    we write $f = O(g)$, when there is a positive constant $C \in \R$ such that $f(n) < C g(n)$ for all sufficiently large values of $n$.
    The notation $f = \Omega(g)$ means there is a positive constant $c > 0$ such that $f(n) > c g(n)$
    for all sufficiently large values of $n$, and $f = \Theta(g)$ means both $f = O(g)$ and $f = \Omega(g)$.
    We write $f=o(g)$ if $\lim_{n \to \infty}\frac{f(n)}{g(n)} = 0$.

\section{Edge Cover Time of Finite Graphs}\label{sec:edge_cover_time}

    In GRW the choice of the next move depends on the history of the walk with
    respect the adjacent edges of the current vertex.
    Hence, it seems more natural to ask about the edge cover time, rather than vertex cover time.
    We show that for some common families of graphs the greedy walk covers the edges
    asymptotically faster than the simple random walk.

    Let $G=(V,E)$ be a connected undirected graph on $n$ vertices.
    Denote by $C_E(G)$ the edge cover time of GRW, i.e., the number of steps it takes for GRW
    to traverse all edges of $G$. Note that since the graph $G$ is finite,
    the edge cover time $C_E(G)$ is a.s.\ finite.

    The basic idea behind the analysis is as follows.
    Divide the random discrete time interval $[0,C_E(G)]$ in two (random) parts:
    \begin{enumerate}
        \item The greedy part: all times in which the walker is at a vertex, that has an adjacent edge, yet to be covered,
                i.e., all times $t \in [0,C_E(G)]$ such that $\{ X_t, X_{t+1} \} \notin H_t$.
        \item The simple part: all times in which the walker is positioned at a vertex
                all of whose adjacent edges have already been covered previously,
                i.e., all times $t \in [0,C_E(G)]$ such that $\{ X_t, X_{t+1} \} \in H_t$.
                In these times the choice of the next move has the same distribution as the one of a simple random walk.
    \end{enumerate}

    Roughly speaking, the GRW typically looks as follows.
    It starts at $t_0 = 0$ in a greedy time part.
    This time part lasts until
    reaching at time $s_1$ a vertex $v_1$ whose all adjacent edges have already been covered.
    We say in this situation that the walk got stuck.
    This means that the last step before getting stuck covered the last edge touching $v_1$.
    Since at time $s_1$ all edges touching $v_1$ have already been covered,
    the walker picks an edge at random among these edges.
    In other words the walk is now in a simple time part, started at time $s_1$.
    This time part lasts until the walker reaches at time $t_1$ a vertex $u_1$,
    that has an adjacent edge which has not been covered yet.
    By definition, the next step will belong to a greedy part, and will continue until reaching at time $s_2$
    some vertex $v_2$, whose all adjacent edges have already been covered, thus starting the second simple part.
    The walk continues in this way until all edges are covered, and then becomes a simple random walk.

    Formally, define the times $t_0,s_1,t_1,s_2,t_2,\dots,s_n$ recursively,
    where the intervals $[t_{i-1},s_i)$ denote the $i$th greedy part, and
    the intervals $[s_i,t_i)$ denote the $i$th simple part of the walk.
    \[
        t_0  =  0,
    \]
    \[
        s_{i+1}  =
        \begin{cases}
            \inf  \{t_i < t \leq C_E(G)  : J_t(X_t)  = \emptyset\} & \mbox{ if there is such }t \\
            C_E(G) & \mbox{ otherwise },
        \end{cases}
    \]
    \[
        t_{i+1}  =
        \begin{cases}
             \inf \{s_{i+1} < t \leq C_E(G) : J_t(X_t)  \neq \emptyset\} & \mbox{ if there is such }t \\
            C_E(G) & \mbox{ otherwise }.
        \end{cases}
    \]

    We say the walk \emph{got stuck} at time $t$ if $t = s_i$ for some $i \in \N$.
    It should be clear from the description, that the vertices $X_{s_i}$ must all be distinct,
    as $X_{s_i}$ is the $i$th time that the walk got stuck, and it is impossible to
    get stuck in the same vertex twice. Therefore, it is enough to define the times $t_i$ and $s_i$
    only for $i \leq n$ (where $n$ denotes the number of vertices in $G$).
    This gives a random partition $(0=t_0 < s_1 < t_1 < s_2 < t_2 <\dots < t_{k-1} < s_k = t_k = \dots =s_n=C_E(G))$
    of the time segment $[0, C_E(G)]$, where the random variable $k \leq n$ is the first $i$
    for which $s_i=C_E(G)$, i.e., all edges of $G$ are covered.

    Note that the total time the walker spends in the greedy parts equals to the number of edges $|E|$,
    implying the following expression on the edge cover time.
    \[
        C_E(G) = |E| + \sum_{i=1}^{n} (t_i - s_i).
    \]
    By linearity of expectation we have the following simple expression
    for the expected edge cover time, which will be the key formula in our proofs.
    \begin{proposition}[Key formula]
        \label{prop:formula}
        Let $G = (V,E)$ be a graph with $n$ vertices, and let $t_0,s_1,t_1,s_2,t_2,\dots$
        be random times as above. Then, the expected edge cover time of GRW on $G$ is
        \begin{equation}\label{eq:formula}
            \E [C_E(G)] =  |E| +  \sum_{i=1}^{n} \E[t_i - s_i].
        \end{equation}
    \end{proposition}
    Thus, in order to bound $\E[C_E(G)]$, it is enough to bound the expected total size of
    all simple parts, i.e., $\E[\sum_{i=1}^k (t_i - s_i)]$.
    In order to apply Proposition~\ref{prop:formula}, the following notation will be convenient.
    For $i=1,\dots,n$ let
    \[
        B_i = \{v \in V : J_{s_i}(v) = \emptyset\},
    \]
    be the set of vertices, all of whose adjacent edges are covered by time $s_i$
    ($B$ stands for "bad"; if the walker is in some vertex in $B$, then the next
    step will be along an edge that has already been crossed, thus increasing the edge cover time).
    By the definition of $s_i$ and $t_i$, we note that
    $B_i = \{v \in V : J_{t}(v) = \emptyset\}$ for every $t \in [s_i, t_i]$.
    Note also that $B_i \seq B_j$ for all $i < j$, and the vertex
    $v_j = X_{s_j}$ in which the walker got stick at time $s_j$
    does not belong to  $B_i$ for $i < j$, as at any time $t < s_j$
    the vertex $v_j$ still had an adjacent edge which has not been covered yet.
    Thus the containment $B_i \subsetneqq B_j$ is strict for all $i < j \leq k$
    i.e., the sets $B_i$ form a strictly increasing chain until it stabilizes at $B_k = V$:
    \begin{equation}\label{eq:B_i's}
        B_1 \subsetneqq B_2 \subsetneqq \dots \subsetneqq B_k = B_{k+1} = \ldots = B_n = V.
    \end{equation}
   In particular,
   \begin{equation}\label{eq:B_i and k}
        |B_i|<n     \mbox{ if and only if }i<k.
    \end{equation}

    Conditioned on $B_i$ and $X_{s_i}$, the length of the time segment $[s_i, t_i]$
    is distributed as the escape time of a simple random walk from $B_i$, when started at $X_{s_i}$.
    That is, conditioned on $B_i$ and $X_{s_i}$, the random variable $(t_i - s_i)$ has the same distribution as $T(X_{s_i},B_i)$, where
    \[
        T(v,B) = \min\{t: Y_t \notin B | Y_0 = v\},
    \]
    and $Y_0,Y_1,\dots$ is a simple random walk on $G$ started at $Y_0 = v$.
    By applying known bounds of the expected escape time of SRW
    we shall use Proposition~\ref{prop:formula} to upper bound the expected edge cover time of GRW.

\subsection{The Complete Graph}\label{subsec:cluque}

    We prove in this section that for the complete graph with $n$ vertices
    the expected edge cover time is $(1 + o(1)){n \choose 2}$.
    Specifically, we prove the following result.
    \begin{theorem}\label{thm:clique}
        For any rule $\RR$ the expected edge cover time of GRW on $K_n$ is bounded by
        \[
            \E[C_E(K_n)] \leq |E| + (1 + o(1))n \log n.
        \]
    \end{theorem}
    This is an improvement over the $\Theta(n^2 \log n)$ time of the SRW,
    which follows by using the coupon collector argument.

    \begin{proof}
    Consider the complete $n$-vertex graph $G = K_n$.
    The proof relies on the following simple observation.
    For any set of vertices $B \seq V$, the escape time of SRW from $B$
    depends only on the size of $B$, and has geometric distribution.
    Specifically, for each $i =1,\dots,n$, the quantity $t_i-s_i$
    \emph{conditioned on} $B_i$ is distributed geometrically:

   \begin{equation}\label{eq:geom}
        t_i - s_i \sim
        \begin{cases}
            G(\frac{n - |B_i|}{n - 1})  &   \mbox{ if } |B_i|<n \\
                    0                   &   \mbox{ otherwise}.
    \end{cases}
   \end{equation}
   Denote by $T_i$ the expected escape time from the subset $B_i$. Then,
    \begin{equation}\label{eq:expectation of geom}
        T_i = \E (t_i-s_i | B_i) =
        \begin{cases}
        \frac{n-1}{n-|B_i|}     &   \mbox{ if $|B_i|<n$} \\
                0               &   \mbox{ otherwise}.
        \end{cases}
    \end{equation}
    By averaging over $B_i$'s, the quantity $\sum_{i=1}^n \E [t_i-s_i]$ is equal to
    \[
        \sum_{i=1}^n \E [t_i-s_i]
            =   \sum_{i=1}^n \E [\E (t_i-s_i | B_i)]
            =   \sum_{i=1}^n \E [T_i]
            =   \E \left[ \sum_{i=1}^{k-1}\frac{n-1}{n-|B_i|} \right],
    \]
    where the last equality follows from
    linearity of expectation, together with~\eqref{eq:expectation of geom}.
    In order to bound the sum in the expectation, let $b_i = |B_i|$, and note
    that we have an increasing sequence of natural numbers $b_1<b_2<\dots<b_k$
    so that $b_1 \geq 1$ and $b_k=n$ for some $k \leq n$.
    For any such sequence it holds that
    \begin{equation}\label{ineq:bounding bi}
        \sum_{i=i}^{k-1}\frac{n-1}{n-b_i}\leq  \sum_{i=i}^{n-1} \frac{n-1}{n-i}.
    \end{equation}
    To see this note that all summands are positive, and each one on the left
    hand side of the inequality, appears also on the right hand side.
    Therefore, we can upper bound the quantity $\sum_{i=1}^n \E [t_i-s_i]$ by
    \[
        \sum_{i=1}^n \E [t_i-s_i]
          \leq  \sum_{i=1}^{n-1} \frac{n-1}{n-i}
            =   (1+ o(1))n \log n
    \]
    Applying Proposition~\ref{prop:formula} gives the desired result.
    \end{proof}

    \begin{remark}
    We conjecture that if the rule in the greedy part is $\RRand$
    (in which an edge is chosen uniformly at random among
    the adjacent unvisited edges of the current vertex),
    then for odd values of $n$, i.e., when the degree is even,
    the overhead for clique is $O(n)$, i.e., $\E[C_E(K_n)] \leq |E| + O(n)$.
    For a related discussion see Section~\ref{sec:remarks and open problems}.
    \end{remark}

\subsection{Expander graphs}\label{subsec:expanders}

    We apply the same method as in the previous section on expander graphs.
    Let $G = (V,E)$ be a $d$-regular graph on $n$ vertices and
    let $A = A(G) \in \{0,1\}^{V \times V}$ be its normalized adjacency matrix,
    namely
    \[
        A(u,v) =
        \begin{cases}
            1/d       &  (u,v) \in E \\
            0       &  (u,v) \notin E.
        \end{cases}
    \]
    It is a standard fact that $A$ has real eigenvalues, all lying in the interval $[-1,1]$.
    Denote the eigenvalues by $1 = \lambda_1 \geq \lambda_2 \geq \dots \geq \lambda_{n} \geq -1$,
    and let $\lambda(G)$ be the spectral radius of $G$, defined as
    \[
        \lambda(G) = \max_{i=2, \dots, n} |\lambda_i|
    \]
    We say a $d$-regular graph $G$ is a $(n, d, \lambda)$-expander,
    if $\lambda(G) < \lambda < 1$
    (for more details see the excellent survey \cite{HLW06}).

    We are able to show that for $d = \Omega(\log n)$,
    the expected edge cover time of the GRW is linear in the number of edges.
    This is faster than a simple random walk, which covers the edges
    in $\Omega ( |E| \log |E| )$ steps, as mentioned in the introduction.
    Specifically, we prove the following theorem.

    \begin{theorem}\label{thm:spectral_bound}
        Let $G$ be a $(n, d, \lambda)$-expander graph.
        Then, for any rule $\RR$ the expected edge cover time is
        \[
            \E[C_E(G)] \leq |E| + O \left( \frac{n \log n}{1-\lambda} \right).
        \]
        In particular, for an expander with $d = \Omega(\log n)$ the expected edge cover time of the GRW is linear in the number of edges.
    \end{theorem}

    \begin{proof} The key observation here is that, as in the case of the complete graph,
    $\E(t_i - s_i| B_i)$ can be bounded in terms of the \emph{size} of $B_i$,
    independently of its structure. We use the following lemma of Broder and Karlin.

    \begin{lemma}[{\cite[Lemma 3]{BroderKarlin89}}]\label{lemma:BK}
        Let $G$ be an $(n,d,\lambda)$-expander and let $S \subsetneqq V$ be a non-empty set of vertices.
        Consider a simple random walk $Y_0, Y_1 \dots$ on $G$, starting at some $v \in S$ (i.e., $Y_0 = v$).
        Let $T(v,S)$ be the escape time of the walk from $S$ when started from $v$.
        Then
        \[
            \E[T(v,S)] \leq \frac{C}{1-\lambda} \left( \log n + \frac{n}{n-|S|} \right)
        \]
    \end{lemma}
    Denoting by $T_i$ the expected escape time from the subset $B_i$,
    by Lemma~\ref{lemma:BK}, for all $i = 1, \dots, n$ we have
    \begin{equation}\label{eq:expanders bound}
        T_i := \E (t_i-s_i | B_i) \leq
        \begin{cases}
            \frac{C}{1-\lambda} \left(\log n  + \frac{n}{n-|B_i|} \right)   &   \mbox{ if $|B_i|<n$} \\
            0       &       \mbox{ otherwise }
        \end{cases}
    \end{equation}
    for some absolute constant $C \in \R$.
    In order to upper bound $\sum_{i=1}^n \E [t_i-s_i]$
    we apply an analysis similar to that in the proof of Theorem~\ref{thm:clique}.
    Specifically, by averaging over the $B_i$'s, the quantity $\sum_{i=1}^n \E [t_i-s_i]$ equals to
    \[
        \sum_{i=1}^n \E [t_i-s_i]
            =   \sum_{i=1}^n \E [\E (t_i-s_i | B_i)]
            =   \sum_{i=1}^n \E [T_i]
            =   \E \left[ \sum_{i=1}^n T_i \right],
    \]
    where the last equality follows from linearity of expectation.
    Using~\eqref{eq:expanders bound}
    we obtain
    \begin{eqnarray*}
        \sum_{i=1}^n \E [t_i-s_i]
            & \leq   &   \E \left[ \sum_{i=1}^{k-1} \frac{C}{1-\lambda} \left(\log n  + \frac{n}{n-|B_i|} \right) \right] \\
            & \leq  &   \frac{C}{1-\lambda} \cdot n \log(n) + \frac{C}{1 - \lambda} \E \left[ \sum_{i=1}^{k-1}\frac{n}{n-|B_i|} \right] \\
            & \leq  &   O \left( \frac{n \log n}{1 - \lambda} \right),
    \end{eqnarray*}
    where the bound $\sum_{i=1}^{k-1}\frac{n}{n-|B_i|} \leq O(n \log(n))$
    in the last inequality follows using the same proof as~\eqref{ineq:bounding bi}.
    Using Proposition~\ref{prop:formula}, we have
    \[
        \E [C_E(G)] \leq |E| + \sum_{i=1}^n \E [t_i-s_i]
                     = |E| + O \left(\frac{n \log n}{1 - \lambda} \right),
    \]
    which completes to proof of the theorem.
    \end{proof}

    Next, we strengthen Theorem~\ref{thm:spectral_bound} by showing that for
    constant degree expanders with logarithmic girth whose vertices have even
    degrees, the expected edge cover time is linear in the number of vertices.
    Recall that girth of a graph $G$, denoted by $\girth(G)$ is the minimal
    length of a cycle in $G$.
    This result is claimed in~\cite{BCF} without proof.

    \begin{theorem}\label{thm:expanders log(n)-girth even deg}
        Let $G$ be a $(n, d, \lambda)$-expander graph
        such that $d \in \N$ is even, and $\girth(G) = g$.
        Then, for any rule $\RR$ the expected edge cover time is
        \[
            \E[C_E(G)] \leq |E| + O \left( |E| \cdot \frac{\log(n)}{(1-\lambda) g} \right).
        \]
        In particular, if $G=(V,E)$ is an expander of constant even degree
        with $\girth(G) = \Omega(\log(n))$,
        then the expected edge cover time of the GRW is linear in the number of vertices.
    \end{theorem}

    The proof relies on following simple observation.
    Suppose that the $i$th greedy part starts at some vertex $v = X_{s_i}$.
    Then, using the fact that all degrees of $G$ are even,
    we conclude that this greedy part will end at the same vertex $v$.
    Indeed, by an Euler-path type argument if a vertex has even degree
    and the walker entered this vertex along a new edge that has not
    been visited so far, then by parity there must be another unvisited edge for the
    walker to leave the vertex. In particular, the range covered by each greedy
    part forms a (not necessarily simple) cycle.
    We summarize this observation below:
    \begin{observation}\label{obs:greedy part cycles}
        If the all degrees of a graph $G= (V,E)$ are even, then in each greedy time
        part $[t_i,s_{i+1}]$ it holds that $X_{t_i} = X_{s_{i+1}}$, i.e., every greedy part
        ends at the same vertex it has started from.

        Therefore, since at the greedy time parts the walker crosses no edge twice,
        in each greedy part $[t_i,s_{i+1}]$ the walker traverses
        along some (not necessarily simple) cycle, and thus the number of steps
        in each greedy time part is at least $\girth(G)$.
    \end{observation}

    We now turn to the proof of Theorem~\ref{thm:expanders log(n)-girth even deg}.
    \begin{proof}
        Using the proof of Theorem~\ref{thm:spectral_bound} the expected edge cover time
        can be upper bounded by
        \begin{equation}\label{eq:spectral bound}
            \E[C_E(G)] = |E| + O\left( \frac{1}{1-\lambda} \right) \cdot
                            \E \left[ \sum_{i=1}^{k-1} \left(\log n  + \frac{n}{n-|B_i|} \right) \right].
        \end{equation}
        By Observation~\ref{obs:greedy part cycles} it follows that
        the random number $k$ of greedy parts is upper bounded by $\frac{|E|}{g}$.
        Therefore, the term $\log(n) \E[k]$ is bounded from above by
        $\frac{\log(n) |E|}{g}$.

        In order to bound the terms $\frac{n}{n-|B_i|}$ note that
        for all $i \leq k$, it holds that $k \leq i + \frac{d \cdot (n - |B_i|)}{g}$.
        Indeed, if in time $s_i$ the number of vertices all of whose adjacent edges
        have already been covered is $|B_i|$, then the number of edges that have
        not been traversed so far is at most $d \cdot (n - |B_i|)$, and hence,
        by the assumption on the girth of $G$, the number of remaining greedy parts is at most
        $\frac{d \cdot (n - |B_i|)}{g}$. Therefore, for all $i \leq k$ we have
        \[
            \frac{n}{n - |B_i|} \leq \frac{dn}{(k-i) \cdot g} = \frac{|E|}{(k-i) \cdot g}.
        \]
        By \eqref{eq:spectral bound} we have
        \begin{eqnarray*}
            \E[C_E(G)]
                    &   =   &   |E| +  O\left( \frac{1}{1-\lambda} \right) \cdot \E[k \log(n)]
                            + O \left( \frac{1}{1-\lambda} \right) \cdot \left( \sum_{i=1}^{k-1} \frac{n}{n-|B_i|} \right) \\
                    &   =   &   |E| +  O \left( |E| \cdot \frac{\log(n)}{(1-\lambda) g} \right)
                                    + O \left( \frac{1}{1-\lambda} \right) \cdot \left( \sum_{i=1}^{k-1} \frac{|E|}{(k-i) \cdot g} \right) \\
                    & \leq  &   |E| +  O \left( |E| \cdot \frac{\log(n)}{(1-\lambda) g} \right),
        \end{eqnarray*}
        where the last inequality uses the assumption that $k \leq n$ and the facts that $\sum_{i=1}^k \frac{1}{k-i} \leq \log(k)$.
        Theorem~\ref{thm:expanders log(n)-girth even deg} follows.
    \end{proof}

    We show below that the assumption that graph has logarithmic girth
    in Theorem~\ref{thm:expanders log(n)-girth even deg} is necessary.
    Specifically, we present a 6-regular expander graph  graph $G$,
    and a rule $\RR$, such that GRW with the rule $\RR$
    coves all the edges of $G$ in expected time $\Omega(n \log(n))$.
    In fact, the graph $G$ satisfies an additional property,
    that every vertex of $G$ is contained in some induced cycle of logarithmic length.
    This should be compared with the result of Berenbrink et al.~\cite{BCF} who have shown
    that if $G$ is an even degree expander such that every vertex of $G$ is contained
    in some induced cycle of logarithmic length, then the expected \emph{vertex}
    cover time by GRW is linear for any rule $\RR$.
    This shows a gap between the edge cover time and the vertex cover time of GRW.

    \begin{theorem}
        For every $n = 0 \pmod{3}$ there exists a 6-regular expander graph $G = (V,E)$ with $|V|=n$
        vertices such that every vertex of $G$ is contained in an induced cycle of logarithmic length,
        and there exists a rule $\RR$ such that the expected edge cover time of $G$
        by GRW with the rule $\RR$ is $\Omega(n \log(n))$.
    \end{theorem}

    \begin{proof}
        Let $H=(U,F)$ be a 4-regular expander graph on $n/3$ vertices such that
        every vertex of $G$ is contained in an induced cycle of length $\eps\log(n)$
        for some constant $\eps>0$.%
        \footnote{Such graph can be obtained by choosing a random 4-regular graph.
        For reference see~\cite[Chapter II.4]{randomgraphs}.}
        Define a graph $G=(V,E)$ to be the cartesian product of $H$ with the graph $K_3$.
        Namely, the vertices of $G$ are $V = U \times \{1,2,3\}$
        and $((u,i),(u',j)) \in E$ if and only if either (1) $(u,u') \in F$ and $i=j$,
        or (2) $u=u'$ and $i \neq j$.
        By the properties of $H$, the graph $G$ is a 6 regular expander and
        it satisfies the property that every vertex of $G$ is contained
        in some induced cycle of length at least $\eps \log(n)$.

        The vertices of $G$ are naturally partitioned into 3 subsets
        $V = V_1 \cup V_2 \cup V_3$ where $V_i = \{ (U,i) : u \in U\}$ for $i=1,2,3$.
        The rule $\RR$ is defined so that the first greedy part will cover
        all edges of the form $((u,i)(v,i))$ for all $(u,v) \in F$ and $i \in \{1,2,3\}$.
        Assume now that GRW starts from some arbitrary vertex $(u_0,1) \in V_1$.
        The walker walks along some Eulerian cycle on $V_1$,
        thus, covering all edges induced by $V_1$.
        Indeed, this can be done, as the graph induced by $V_1$ is isomorphic to $H$,
        and hence its vertices have even degrees.
        After completing the cycle in $V_1$, and returning to the initial vertex $(u_0,1)$,
        the walker moves to $(u_0,2)$, performs a walk along some Eulerian cycle on $V_2$,
        and returns back to $(u_0,2)$.
        Similarly, the walker, then, moves to $(u_0,3)$,
        covers all edges induced by $V_3$, and returns to $(u_0,3)$.
        Finally, the walker moves back to $(u_0,1)$, and gets stuck for the first time.
        Note that at this point all edges induces by each of $V_i$'s have already been
        covered by GRW, and the remaining edges form disjoint triangles of the form
        $\{(u,1),(u,2),(u,3)\}$ induced by each of the vertices $u \in U \setminus \{ u_0 \}$.
        Hence, each subsequent greedy part will consist of 3 steps, covering one triangle at each part,
        and the order is defined by the first time that SRW will reach some
        vertex of a triangle $\{(u,i) : i = 1,2,3\}$.
        Noting that SRW on $G$ induces a lazy-SRW on $H$ (where a lazy step in $H$ corresponds
        to a step from $(u,i)$ to $(u,j)$ in $G$), it follows that in order to cover all triangles,
        lazy-SRW needs to cover all the vertices of a copy of $H$.
        Since by the theorem of Fiege~\cite{Feige95lowerbound}
        the expected vertex cover time of every graph by SRW is at least $\Omega(n \log(n)$,
        this bound also holds for the edge cover time of $G$.
        This completes the proof of the theorem.
    \end{proof}

\subsection{Hypercube $\{0,1\}^d$}\label{subsec:hypercube}

    The hypercube graph $G = (V,E)$ is a graph, whose vertices are
    $V = \{0,1\}^d$ and $(u,v) \in E$ if and only if $d(u,v) = 1$,
    where $d(\cdot,\cdot)$ is the Hamming distance between two strings.
    We show that for even dimension $d$ the edge cover time of the hypercube
    is linear in the number of edges.

    \begin{proposition}\label{prop:hypercube}
        Let $d \in \N$ be even, and let $Q_d = (V,E)$ be the $d$-dimensional hypercube graph.
        Then, for any rule $\RR$ the expected edge cover time of $Q_d$ is bounded by
        \[
            \E[C_E(Q_d)] = O(|E|).
        \]
        \end{proposition}

    \begin{proof}
    The proposition follows from the following lemma.

    \begin{lemma}\label{lemma:SRW cover time bound}
        Let $G = (V,E)$ be a graph whose vertices have even degrees.
        Suppose that for the graph $G$ the expected \emph{vertex} cover time of SRW
        is $C$. Then, the expected \emph{edge} cover time of GRW of $G$ is at most
        \[
            \E[C_E(G)] \leq |E| + C.
        \]
    \end{lemma}
    Since the number of edges in $Q_d$ is $|E|=\frac{1}{2} d \cdot 2^d$,
    and using the fact that the expected vertex cover time of the hypercube by SRW
    is $C = O(d \cdot 2^d)$, Proposition~\ref{prop:hypercube} follows
    by Lemma~\ref{lemma:SRW cover time bound}.
    \end{proof}

    We now turn to the proof of Lemma~\ref{lemma:SRW cover time bound}.

    \begin{proof}[Proof (of Lemma~\ref{lemma:SRW cover time bound})]
        The proof proceeds by coupling between a SRW and a GRW so that
        the number of steps made by the GRW is larger than the number of
        steps made by SRW by at most $|E|$.

        As observed above, in Observation~\ref{obs:greedy part cycles}
        for graphs of even degrees we have $X_{t_i} = X_{s_{i+1}}$ for all $i \leq k$,
        i.e., every greedy part finishes at the same vertex that it started from.
        This implies that the simple parts can be concatenated,
        as the end of the $i$th simple part is $X_{t_i}$, and the
        beginning of the $(i+1)$th part is $X_{s_{i+1}}$.
        The coupling between the SRW and the GRW is the natural one, where the SRW
        performs all the steps that the GRW makes in its simple parts.
        Clearly, the number of steps made by the GRW is larger than the number of
        steps made by SRW by at most the total number of steps made in the greedy parts,
        which is bounded by $|E|$.

        Observe that whenever the SRW reaches some vertex $v$, it is either the case that
        (1) all edges adjacent to $v$ have already been covered by GRW, or
        (2) the vertex $v$ is the last vertex in the current simple part,
        and thus, using the property $X_{t_i} = X_{s_i}$ for all $i$,
        the next greedy part will cover all edges adjacent to $v$.
        This implies that by the time the SRW covers all vertices of $G$, the GRW
        has either already covered all edges of $G$, or will do so in the number greedy part.
        Therefore, the \emph{edge} cover time of GRW is larger than
        the \emph{vertex} cover time of SRW by at most $|E|$.
        This completes the proof of the lemma.
    \end{proof}

    We also remark (without a proof) on the edge cover time
    of a generalization of the hypercube graph.
    \begin{remark}
    Define a generalization of the hypercube by
    connecting two vertices in $\{0,1\}^d$ if the distance between them
    is at most some parameter $\ell \geq 2$.
    Specifically, for $\ell \geq 2$, let $H_d^{(\leq \ell)}=(V,E_\ell)$,
    where $V = \{0,1\}^d$ and $(x,y) \in E$ iff $d(x,y) \leq \ell$.
    Denoting the number of vertices in the graph by $n = 2^d$,
    the spectral radius of $H_d^{(\leq \ell)}$ is bounded
    from above by $\lambda \leq 1 - \frac{\ell}{\log n}$.
    Therefore, by Theorem~\ref{thm:spectral_bound} for $\ell \geq 2$
    the expected edge cover time of GRW on $H_d^{(\leq \ell)}$
    is $|E_\ell| + O ( n \log^2 n )$,
    where the constant in the $O()$ notation depends on $\ell$.

    Noting that the number of edges in $H_d^{\leq \ell}$ is
    $|E_\ell| = O(n \cdot \log^{\ell}n)$, this implies that
    for $\ell = 2$ the edge cover time is linear in
    the number of edges $|E_2| = O(n \cdot \log^{2}n)$,
    and for $\ell \geq 3$ the edge cover time is $(1+o(1)|E_\ell|$.
    \end{remark}
\subsection{$d$-regular trees}\label{subsec:d-regular_tree}

    In this section we provide an upper bound for the edge cover time of GRW on trees.
    We are able to describe the behavior of GRW quite accurately, and subsequently
    provide a tight bound on the cover time.

    \begin{theorem}\label{thm:trees}
        Let $G = (V,E)$ be a tree rooted at a vertex denoted by $r$, such that $\deg(r) \geq 2$.
        For any $v \in V$ denote by $T_v$ the subtree rooted at $v$
        and let $|T_v|$ denote the number of edges in $T_v$.
        Then, for any rule $\RR$ the GRW edge cover time of $G$ is
        \[
            \E[C_E(G)] = |E| + O \left( \sum_{u \in G \setminus \{r\}} |T_u| \right).
        \]
        If the rule for GRW is $\RRand$, then there is a matching lower bound, namely
        \[
            \E[C_E(G)] = |E| + \Theta \left( \sum_{u \in G \setminus \{r\}} |T_u| \right).
        \]
    \end{theorem}

    The following corollary in immediate from Theorem~\ref{thm:trees}.
    \begin{corollary}\label{cor:d-reg trees}
        If $G$ is a $d$-regular tree with $n$ vertices, then the expected edge cover time is $O(n \log_d n)$.
    \end{corollary}

    Comparing Corollary~\ref{cor:d-reg trees}
    to the cover time of SRW on $d$-regular trees, we again see an asymptotic speed-up
    over the $\Theta (n \log^2_d n )$ time of the SRW \cite{aldous1991random}.

    \begin{proof}
        In order to use the tree structure of the graph, let us first
        give an overview of the behavior of GRW on trees.
        The walker starts at the root $r$ and goes down greedily
        (i.e., an unvisited edge is traversed in every new step), until reaching a leaf.
        Since it got stuck at a leaf, it performs a simple random walk until reaching
        its lowest ancestor with an adjacent edge that has not been covered yet.
        The non-covered edge is necessarily from the ancestor to one of its children,
        (as its parent has been already visited on the way down).
        The walker continues by moving down greedily until reaching another
        leaf not covered thus far, and then performs an SRW until reaching again
        its lowest ancestor with a child that has not been visited thus far by the walk.
        The walk continues in the same manner until covering all edges,
        getting stuck only at the leaves. In fact the walk gets stuck exactly
        once in each leaf, and the time $C_E(G)$ is the time when walker visits
        the last leaf of the tree. Note that when visiting some vertex $v$, the walk will
        cover the entire subtree of $v$ before returning to $v$'s parent.
        This property is what makes the cover time of GRW asymptotically faster
        than the cover time of SRW.

        The order in which the vertices are visited for the first time
        defines some preorder traversal on the tree (first the root, then the subtrees),
        where for each vertex the order of the subtrees is chosen according to the rule $\RR$.
        We observe that the vertices $(X_{s_1},X_{s_2},\dots, X_{s_k})$ define some order on the leaves of the tree,
        induced by the preorder traversal as described above
        (and in particular, $k$ equals to the number of leaves).
        In addition, for every $i < k$, the vertex $X_{t_i}$ is the lowest ancestor of $X_{s_i}$
        such that at time $s_i$ not all of its descendants have been visited by the walk.
        Hence, $\E[t_i - s_i]$ equals the expected time it takes for the simple random walk starting at $X_{s_i}$ to visit
        this ancestor.
        This implies that for every edge $(u,v)$, where $u$ is the parent of $v$,
        there is at most one $i \in \{1,\dots,k\}$
        such that the edge $(u,v)$ lies on the shortest path from $X_{s_i}$ to $X_{t_i}$.
        Therefore, if $w$ is the leaf where the walk got stuck for the $i$th time, that is, $X_{s_i} = w$,
        and $v$ is its lowest ancestor whose subtree is not covered yet, then the expected time to reach
        $v$ staring from $w$ is
        \[
            \E[t_i - s_i] = H(w,v) = \sum_{(u_1,u_2) \in P_{(w,v)}} H(u_1,u_2),
        \]
        where $H(x,y)$ denotes the expected number of steps required for SRW starting at $x$
        to visit $y$, and the sum is over all edges on the shorted path from $w$ to $v$
        (using the convention that the edge $(u_1,u_2)$ means that $u_2$ is a parent of $u_1$.

        Going over all leaves in in the graph, and using the observation
        that the walk gets stuck in each leaf exactly once (stopping at the last
        visited leaf at time $s_k$), and finishing the corresponding simple part
        at the lowest ancestor whose tree has not been covered yet, we observe that
        for each $i < k$ the shortest paths from $X_{s_i}$ to $X_{t_i}$
        are disjoint. Furthermore, the union of all these paths covers all edges of the graph
        except for the path from the last covered leaf, denoted by $l = X_{s_k}$, to the root of the tree.
        Let us denote by $P_{(r,l)}$ be the shortest path from $l$ to $r$.
        Then
        \begin{equation}\label{eq:sumH(v,u)}
            \E \left[\sum_{i=1}^k (t_i - s_i) \right] = \E \left[ \sum_{(u,v) \in E \setminus P_{(r,l)}} H(u,v)  \right]
                \leq \sum_{(u,v) \in E} H(u,v),
        \end{equation}
        where $H(v,u)$ denotes the expected number of steps, required for SRW starting at $v$ to visit $u$ for the first time,
        and the summation is over all edges $(u,v)$, where $v$ is the parent of $u$.

        It is well known (see e.g., \cite[Lemma 1]{AKLLR79}) that
        if $(u,v)$ is an edge in a tree, then
        \[
            H(u,v) = 2|T_u| + 1 \leq 3|T_u|.
        \]
        Proposition~\ref{prop:formula} together with~\eqref{eq:sumH(v,u)}, proves the upper bound of the theorem.

        \medskip

        If GRW uses the rule $\RRand$,
        then the subtrees rooted at the children of $r$
        are explored completely one after another (the order of the children is random),
        and the walk will return to $r$ from all but the last subtree.
        Therefore, for each $u$ child of $r$ the subtree rooted at $u$
        is completely explored by GRW with probability $\frac{\deg(r)-1}{\deg(r)}$,
        and hence, every edge of the tree belongs to $P_{(r,l)}$ with probability at most $\frac{1}{\deg(r)}$.
        Therefore, by applying the formula in \eqref{eq:sumH(v,u)} we get
        \begin{eqnarray*}
            \E \left[\sum_{i=1}^k (t_i - s_i) \right]
                &   =   &   \E \left[ \sum_{(u,v) \in E \setminus P_{(r,l)}} H(u,v)  \right] \\
                &  \geq &   (1 - \frac{1}{\deg(r)}) \sum_{(u,v) \in E} H(u,v) \\
                &  \geq &   (1 - \frac{1}{\deg(r)}) \sum_{u \in G \setminus \{r\}} 2|T_u|.
        \end{eqnarray*}
        This completes the proof of the theorem.
    \end{proof}

\section{Greedy Random Walk on $\Z^d$}\label{sec:Z^d}

    In this section we study the behavior of GRW on infinite graphs.
    Specifically we ask whether the walk is recurrent or transient in different graphs.
    Obviously, on $\Z$ GRW visits every vertex at most once.
    We show that for $d \geq 3$ the greedy random walk on $\Z^d$ is transient.

    \begin{theorem}\label{thm:Z^d}
        Let $G = (V,E)$ be an infinite graph, where all its vertices are of even degree.
        If the simple random walk on $G$ is transient, then for any rule $\RR$
        the greedy random walk is also transient.

        In particular for $d \geq 3$, the greedy random walk on $\Z^d$ returns
        to the origin only finitely many times almost surely.
    \end{theorem}

    \begin{proof}
    Partition the time $[0,+ \infty ]$
    into two types of parts, greedy parts and simple parts,
    by defining times $t_0 = 0,s_1,t_1,s_2,t_2,\dots\in\N\cup\{+\infty \}$ as follows:
    \[
        t_0  =  0
    \]
    \[
        s_{i+1}  =
        \begin{cases}
                \inf  \{t_i \leq t <+\infty  : J_t(X_t)  = \emptyset\} & \mbox{ if there is such $t$} \\
                +\infty                 & \mbox{ otherwise }
        \end{cases}
    \]

    \[
        t_{i+1}  =
        \begin{cases}
                \inf \{s_{i+1} \leq t <+\infty : J_t(X_t)  \neq \emptyset\} & \mbox{ if there is such $t$} \\
                +\infty & \mbox{ otherwise.}
        \end{cases}
    \]
    (Analogous partition underlies the results in Section \ref{sec:edge_cover_time}.
    The difference here is that the times can have the value $+\infty$.)

    For the reader's convenience we restate Observation~\ref{obs:greedy part cycles} adapted for the case of infinite graphs.
    \begin{observation}\label{obs:greedy part cycles-infinite}
        If all degrees of a graph $G= (V,E)$ are even and $s_{i+1} < \infty$,
        then $X_{t_i} = X_{s_{i+1}}$.
    \end{observation}

        Assume that the event that $s_i$ or $t_i$ equals $+\infty$ for some $i \geq 1$
        and $t_k$ is the first such time has a positive probability.
        Conditioning on this event, the walk remains in a simple part starting from time $s_k$,
        and hence performs a simple random walk from this time onwards.
        Since SRW is transient on $G$, the walk will return to $X_0$ only finitely many times a.s..
        Actually, since the random range $R=\{X_t: 0\leq s_k\}$ is finite, and the SRW is transient, conditioning on $R$,
        the SRW will leave $R$ in finite time a.s., and so $t_k$ is a.s. finite, contradicting the assumption.

        Similarly, if the event that $s_i$ or $t_i$ equals $+\infty$ for some $i \geq 1$
        and $s_k$ is the first such time has a positive probability,
        then, conditioning on this event, the walk is in a greedy part from $t_{k-1}$ onwards.
        In other words, from time $t_{k-1}$ onwards the walker crosses each edge at most once.
        Hence, as the degree of $X_0$ is finite, the maximal number of returns to $X_0$
        is at most $\frac{\deg(X_0)}{2} + \frac{t_{k-1}}{2}$, and in particular a.s.\ finite.

        Assume now that the event that $s_i,t_i<+\infty$ for all $i \geq 1$ has a positive probability, and condition on this event.
        Using the assumption that all vertices of the graph have even degrees
        it follows from Observation~\ref{obs:greedy part cycles-infinite}
        that $X_{t_i} = X_{s_{i+1}}$ for all $i \geq 0$.
        Therefore, for all $i\geq 0$ the walk in time segments $[s_i,t_i]$ and $[s_{i+1}, t_{i+1}]$
        can be concatenated.
        Hence, the walk restricted to time $\bigcup_{i\geq 0} [s_i,t_i]$ is distributed as a SRW on $G$,
        and so, by transience, returns to $X_0$ finitely many times almost surely.
        Since in the overall greedy parts the walker can visit $X_0$ at most $\frac{\deg(X_0)}{2}$ times,
        the entire walk returns to $X_0$ finitely many times a.s.
    \end{proof}

    Note that we strongly used the fact that all vertices in our graph have even degree.
    The following proposition shows a similar result by slightly relaxing this assumption.
    \begin{proposition}
        Let $G = (V,E)$ be a graph obtained from $\Z^3$ by removing
        at most $r^{1 - \eps}$ edges from any box of radius $r$ around the origin for some $\eps > 0$.
        Then the greedy random walk on $G$ is transient.
    \end{proposition}

    The proof generalizes the concatenation argument of Theorem \ref{thm:Z^d}.
    Unlike the previous proof, which relied on the fact that all vertices had even degrees,
    in our case some vertices have odd degrees.
    Hence it is possible that the simple parts cannot be concatenated into one walk.
    However, we can divide the simple parts into classes,
    such that in each class the parts can be concatenated into one simple random walk.
    The proof uses the fact that if there are $r^{1-\epsilon}$ independent
    simple random walkers started at a box of radius $r$ around the origin,
    then the total number of visits at the origin by all the walkers is almost surely finite.
    \begin{proof}
        We start with a time partition $(t_0 = 0,s_1,t_1,s_2,t_2,\dots)$ as in the proof of Theorem \ref{thm:Z^d}.
        Call a vertex $v$ a \emph{new start} if $v = X_{s_i}$ for some $i$ and
        $X_{t_j} \neq v$ for all $j < i$.
        As in the proof of Theorem \ref{thm:Z^d}, the concatenation argument
        implies that every new start vertex must be either the origin or have an odd degree.

        Consider the walk restricted to the segments $[s_i, t_i]$.
        The indices $i \geq 1$ can be partitioned into classes $C_1, C_2, \dots$
        such that in each class $C_j$ the segments $[s_i, t_i]$, $i \in C_j$,
        can be concatenated into one walk that starts with a new start vertex.
        Namely, if $C_j = \{i_1 < i_2 < i_3 < \dots\}$,
        then $X_{s_{i_1}}$ is a new start and
        $X_{s_{i_2}} = X_{t_{i_1}}, X_{s_{i_3}} = X_{t_{i_2}}, \dots$
        Denoting by $m_j = \min C_j$ we have $X_{s_{m_j}}$ is necessarily
        a new start, and therefore is either the origin or a vertex of odd degree.
        Moreover the times $\{s_{m_j} \}_j$ are all distinct.

        For each $C_j$, restricting the walk to times $\bigcup\{[s_i, t_i] : i \in C_j \}$ gives us a
        simple random walk (possibly finite) starting from $X_{s_{m_j}}$.
        Therefore there are at most $O(r^{1 - \eps})$ simple random walks,
        starting from a box of radius $r$ around the origin.
        Using the fact that a random walk in $\Z^3$ starting from a vertex at distance $r$ from the origin hits it with probability $O(1/r)$,
        we conclude that the sum of probabilities of hitting zero converges, when summing over all random walks.
        More precisely, let $P_v$ be the probability that a SRW starting at $v$ reaches the origin,
        and let $ODD$ be the set of all vertices of odd degree. Then
        \[
            \sum_{v \in ODD}P_v
                = \sum_{n = 1}^\infty \sum_{\substack{v \in ODD \\ 2^{n-1} \leq \norm{v} < 2^n }} P_v
                \leq \sum_n (2^n)^{1-\eps} \cdot O(\frac{1}{2^n}) = O \left( \sum_{n = 1}^\infty 2^{ -\eps n} \right)< \infty.
        \]
        Therefore, by the first Borel-Cantelli lemma, a.s. only finitely many of the walks will reach the origin,
        implying that the GRW on this graph is transient.
    \end{proof}

    As a last application of the concatenation argument,
    we show that for any vertex-transitive graph with even degrees the expected number of edges covered by the GRW in $t$ steps
    cannot be asymptotically smaller than that of the SRW.

    \begin{proposition}\label{prop:range}
        Let $G = (V,E)$ be a vertex-transitive graph with vertices of even degree.
        Denote by $\mathcal{N}_W(t)$ the expected number of edges covered by the walk $W$ in $t$ steps.
        Then $\mathcal{N}_{\textrm{GRW}}(t) \geq \half \cdot \mathcal{N}_{\textrm{SRW}}(t)$.
    \end{proposition}

    Note that by vertex transitivity, $\mathcal{N}_W(t)$ is independent of the initial vertex.
    \begin{proof}
        Let $S$ be the random variable denoting the number of steps that GRW spent in the simple parts up to time $t$,
        and let $(t-S)$ to be the number of steps spent in the greedy parts.
        Denote by $\mathcal{N}_{\textrm{GRW}}(t | s)$ the expected number of edges covered by GRW conditioned on the event $S = s$.
        Using the fact that all vertices of $G$ have even degrees for every $s \leq t$ we have
        $\mathcal{N}_{\textrm{GRW}}(t | s) \geq \max(t-s, \mathcal{N}_{\textrm{SRW}}(s))$.
        Indeed, $t-s$ is the number of steps
        in the greedy part, and therefore $\mathcal{N}_{\textrm{GRW}}(t | s) \geq t-s$.
        The inequality $\mathcal{N}_{\textrm{GRW}}(t | s) \geq \mathcal{N}_{\textrm{SRW}}(s)$
        follows by coupling using the concatenation argument
        as in the proof of Theorem \ref{thm:Z^d}.

        If $s < t/2$, then $\mathcal{N}_{\textrm{GRW}}(t | s) \geq t/2 \geq \mathcal{N}_{\textrm{SRW}}(t/2)$.
        Otherwise, we have
        $\mathcal{N}_{\textrm{GRW}}(t | s) \geq \mathcal{N}_{\textrm{SRW}}(s) \geq \mathcal{N}_{\textrm{SRW}}(t/2)$.
        Note that $\mathcal{N}_{\textrm{SRW}}(t) \leq 2 \cdot \mathcal{N}_{\textrm{SRW}}(t/2)$,
        which holds by the Markov property of SRW on a vertex-transitive graph.
        Averaging over $S$ we get
        $\mathcal{N}_{\textrm{GRW}}(t) = \E[\mathcal{N}_{\textrm{GRW}}(t | S)] \geq \half \mathcal{N}_{\textrm{SRW}}(t)$.
\end{proof}

\subsection{$\Z^2$ and the mirror model}\label{sec:Z^2}

    The following observation relating the behavior of GRW on $\Z^2$ to the
    mirror model is due to Omer Angel.

    In the mirror model, introduced by Ruijgrok and Cohen \cite{Ruijgrok}
    a mirror is placed randomly on $\Z^2$ by aligning a mirror along
    either one of the diagonal directions with probability $1/3$ each, or
    placing no mirror with probability $1/3$. A particle moves along the
    edges of the lattice and is reflected by the mirrors according to the law of reflection.
    See, e.g., \cite{BT92} for details.
    A major open problem in this area is to determine whether every orbit is periodic almost surely.
    We claim below that this question is equivalent to determining whether GRW with rule $\RRand$
    is recurrent in $\Z^2$.
    (Recall, in the rule $\RRand$ an edge is chosen uniformly at random among the adjacent unvisited edges
    of the current vertex).

    Let $(X_t)_{t\geq 0}$ be GRW on $\Z^2$ with the rule $\RRand$.
    Then, there exists a coupling between $(X_t)_{t\geq 0}$ and
    the particle motion in the planar mirror model until the first time they return to the origin.
    Indeed, if at time $t \geq 0$ GRW reaches a vertex $X_t$ that we have not visited so far,
    then in both GRW and in the mirror model the next step will be chosen in a non-backtracking
    manner, giving equal probabilities of $1/3$ to each of the adjacent vertices (except for $X_{t-1}$).
    In the mirror model this uniquely defines the alignment of the mirror at vertex $X_t$, and hence the next move of the particle in the next visit to this place, given that the orbit is not periodic: it will go to the unvisited neighboring vertex.
    On the other hand, if at time $t \geq 0$ we reach a vertex $X_t$ that has already been visited previously,
    then the next step is uniquely determined: it is to make a move along the edge that has not been traversed so far.
    This defines a coupling of the two models up to the first returning time to zero.
    \begin{claim}\label{claim:GRW-mirror first return}
        The probability that GRW with the rule $\RRand$ on $\Z^2$ returns to the origin at least once
        is equal to the probability that a particle returns to origin in the planar mirror model.
    \end{claim}

    From Claim~\ref{claim:GRW-mirror first return} we infer the following theorem.
    \begin{theorem}
        GRW with rule $\RRand$ on $\Z^2$ returns to the origin infinitely often almost surely
        if and only if
        the orbit in the mirror model on $\Z^2$ is periodic almost surely.
    \end{theorem}

    \begin{proof}
        Note first that GRW on $\Z^2$ returns to the origin infinitely often
        if and only if every greedy part is finite. Indeed, if there is an infinite
        greedy part, then there are finitely many returns to the origin as
        every vertex is visited at most twice in total in all greedy time parts.
        In the other direction, assume that all greedy time parts are finite.
        Then, by the concatenation argument, which follows by Observation~\ref{obs:greedy part cycles-infinite},
        the simple parts form an infinite subsequence distributed as SRW on
        $\Z^2$ starting at the origin. The latter returns to the origin infinitely
        often almost surely, and hence, so does GRW.
        Therefore, it is enough to show that the orbit in the mirror model on $\Z^2$
        is periodic a.s. if and only if every greedy part is finite.

        Suppose first that every orbit in the mirror model on $\Z^2$ is periodic almost surely,
        and suppose that GRW starts the $i$th greedy part at some time $t_i$.
        Then, condition on the $t_i$ steps
        of GRW so far, defines the orientation of the mirrors in
        the vertices visited up till now. Since, the number of visited
        vertices is finite, it follows that the conditioning is on a non-zero event,
        and so the trajectory of the particle starting from $X_{t_i}$ is a.s. periodic.
        Therefore, by considering the coupling between GRW and the mirror
        model conditioned on that event, analogously to Claim~\ref{claim:GRW-mirror first return},
        it follows that with probability 1 the $i$th greedy part is finite.

        Assume now that every greedy part of GRW is finite.
        Note that by translation invariance it is enough to show that trajectory of a single
        particle starting at the origin is periodic almost surely.%
        \footnote{Indeed, if trajectory of a particle starting at the origin is periodic almost surely,
        then, by translation invariance the trajectory of a particle starting at any vertex
        and moving in any direction is periodic almost surely.
        Thus, by placing, 4 particles at each vertex of the graph, and letting them move
        in the 4 possible directions, it follows that with probability 1
        the trajectory of each of them a periodic, as this this event is an intersection
        of countably many probability 1 events.
        Therefore, trajectory of a particle is periodic almost surely if and only if
        all orbits are periodic almost surely.}
        Indeed, since GRW returns to the origin twice a.s.,
        it follows from the coupling in Claim~\ref{claim:GRW-mirror first return} that the trajectory of a
        particle starting at the origin is periodic almost surely, as required.
    \end{proof}

\section{Remarks and Open Problems}\label{sec:remarks and open problems}

\subsection{Conjecture Regarding Theorem~\ref{thm:clique}}\label{sec:conjectures}

    Observation~\ref{obs:greedy part cycles} used in the proof
    of Theorem~\ref{thm:expanders log(n)-girth even deg}
    seems to be potentially useful for proving stronger
    bounds on the edge cover time of GRW.
    \begin{observation}
        If the all degrees of a graph $G= (V,E)$ are even, then in each greedy time
        part $[t_i,s_{i+1}]$ it holds that $X_{t_i} = X_{s_{i+1}}$, i.e., every greedy part
        ends at the same vertex that it started from.
    \end{observation}

    To illustrate how this observation can be useful, let us consider the GRW
    on the complete graph $K_n$. In the argument we only used the assumption
    that the "bad" sets $B_i$ grow at least by one each time, thus allowing
    us to bound the ``overhead'' by $\E \left[ \sum_{i=1}^{k-1}\frac{n-1}{n-|B_i|} \right]
    \leq \E \left[ \sum_{i=1}^{k-1}\frac{n-1}{n-i} \right] \leq n \log(n)$.
    We suspect, however, that in fact the sets $B_i$ grow linearly in $n$,
    as by the time the walker gets stuck for the first time, i.e., visits
    the starting vertex $n/2$ times, we expect that the number of vertices
    that have already been visited $n/2$ times is linear in $n$.
    The situation, however, becomes more complicated when trying to analyze the set
    $B_2$, as it seems to require some understanding regarding the subgraph
    of $K_n$ that has not been covered by the time $s_1$, when the walker got
    stuck for the first time.
    If this is indeed true, and the sets $B_i$ grow linearly in each step,
    we would obtain a stronger bound $\E[\sum(t_i - s_i)] = O(n)$.
    We make the following, rather bold, conjecture.
    \begin{conjecture}
        The expected edge cover time of GRW on $K_n$ is
        \[
            \E[C_E(K_n)] = |E| + \Theta(n).
        \]
    \end{conjecture}

    An interesting result in this direction is a recent result of Omer Angel
    and Yariv Yaari. They showed that for the complete graph $K_n$ for odd
    values of $n$, i.e., when the graph $K_n$ is of even degree, the expected
    number of unvisited edges in $K_n$ until the first time the walk got stuck
    (i.e., up to time $s_1$) is linear in $n$ \cite{Yaari}.

\subsection{Rules on Vertices Instead of Edges}\label{sec:vertex-GRW}
    In this paper we have considered \emph{edge} cover time of graphs, rather
    than \emph{vertex} cover time. This seems to be a natural quantity to
    analyze due to the transition rule of GRW. A na\"{i}ve modification of GRW
    to speed-up the \emph{vertex} cover time is the following. At each step,
    the walker at vertex $v$ picks an unvisited neighbor of $v$ according
    to some rule and jumps there.
    If all neighbors have already been visited, the next move is chosen
    uniformly at random among the neighbors of $v$. For example, it is
    obvious that in the complete graph $K_n$, this walk covers all vertices in $n$ steps.

    Note that, when the walker is allowed to make some local computations at a vertex,
    and each vertex has the information regarding its neighbors, then one can define
    a rule that will force the walk to perform depth first search on the graph,
    by letting each vertex use only the information regarding its neighbors.
    Such walk crosses at most twice each edge of some spanning tree,
    thus visiting all vertices of the graph in less than $2n$ steps.
\subsection{Open Problems}\label{sec:open-problems}

    In order to avoid trivialities, in the questions below consider GRW with the rule $\RRand$.

    \begin{enumerate}
    \item
        Given a tight bound for the ``overhead'' of GRW on the complete graph.
        Specifically, is it true that $\E[C_E(K_n)] = {n \choose 2} + \Theta(n)$?

    \item
        Show upper bounds on $C_E(G)$ for other families of graphs.
        One interesting example to look at could be the $d$-dimensional torus.

    \item
        It seems also interesting to analyze the GRW on graphs with power-law
        degree distribution. On such graphs there are hubs of very large degrees
        and when visiting them, the GRW is expected to be efficient.

    \item
        Show that for any transitive graph the expected edge cover time of the
        GRW cannot be asymptotically larger than that of the SRW for any finite graph.
        By Proposition \ref{prop:range}, this is true for vertex-transitive graphs of even degree.

    \item
        Give bounds on the expected \emph{vertex} cover time of the GRW for finite graphs.

    \item
        Give bounds on the expected hitting time of GRW for different graphs.

    \item
        Define GRW mixing time and show that GRW mixing time is as fast as that of SRW.
        Here \cite{ABLS06} is relevant, and also \cite{Madras2005} may be found useful.

    \bigskip
        \noindent
        The remaining problems are regarding recurrence/transience of GRW on infinite graphs.

    \item
        Is GRW on $\Z^2$ recurrent? Is GRW diffusive on $\Z^d$, for all $d \geq 2$?

    \item
        Is GRW on the ladder $\Z \times \Z_2$ recurrent?

    \item
        Prove that GRW is transient on any graph that is roughly isometric to $\Z^3$.
        In particular show it for odd degree lattices.

    \item
        Show that GRW is transient on non-amenable infinite graphs.

    \item
        Consider GRW on a vertex transitive graph. Is there a zero-one law for the following events:
        \begin{enumerate}
            \item[(a)] The walker returns to its initial location infinitely often.
            \item[(b)] The walker returns to its initial location at least once.
        \end{enumerate}
        Is it true that (a) happens almost surely if and only if (b) happens almost surely?
    \end{enumerate}

\section{Acknowledgements}
    We are grateful to Itai Benjamini for proposing the model and for many valuable discussions,
    and to Gady Kozma for his useful ideas.
    We thank Shlomo Jozeph for suggesting to use the concatenation argument in Section \ref{sec:Z^d}.
    We also thank Eviatar Procaccia for pointing out Omer Angel's argument relating GRW to the mirror model in $\Z^2$.


\appendix

\end{document}